\documentclass[10pt]{amsart}
\usepackage{latexsym,enumerate}
\usepackage{amsmath,amsthm,amsopn,amstext,amscd,amsfonts,amssymb}
\usepackage[latin1]{inputenc}
\usepackage{epstopdf}
\usepackage{graphicx}
\usepackage[active]{srcltx}
        \newtheorem{Theorem}{Theorem}[section]
\newtheorem{Corollary}[Theorem]{Corollary}
\newtheorem{Definition}[Theorem]{Definition}
\newtheorem{Lemma}[Theorem]{Lemma}
\newtheorem{Proposition}[Theorem]{Proposition}

\newtheorem{remark}[Theorem]{Remark}
\date{}
\catcode`@=11 \@addtoreset{equation}{section}
\renewcommand\theequation{\thesection.\@arabic\c@equation}
\catcode`@=12
\begin{document}

\title[Global dynamics]
{Global dynamics for a  class of reaction-diffusion equations with distributed delay and Neumann condition }
%\author[V. Vitaly]
%{Volpert Vitaly}
%\author[T.M.Touaoula]
%{ Tarik Mohamed Touaoula}
%\thanks{ The authors are partially supported by projects
%A/030893/10, AECI, Spain.}

%\address{ Tarik Mohammed Touaoula
%\hfill \break\indent Laboratoire d'Analyse Non linéaire et Mathématiques Appliquées\hfill \break Département de Mathématiques, Université
%Aboubekr Belka\"{\i}d, Tlemcen, \hfill\break\indent Tlemcen 13000,
%Algeria.} \email{{\tt tarik.touaoula@mail.univ-tlemcen.dz}}

%\address{ Volpert Vitaly
%\hfill \break\indent INRIA Rhône, Institut Camille Jordan UMR
%5208, 43 blvd du 11 novembre 1918, F-69622 Villeurbanne-Cedex,
%France} \email{{\tt vitaly@math.univ-lyon1.fr}}
\author[Tarik Mohammed Touaoula]{Tarik Mohammed Touaoula \\ \tiny{Département de Mathématiques, Faculté des Sciences, Université de Tlemcen, Algérie. Laboratoire d'Analyse Non Linéaire et Mathématiques Appliquées}}


\maketitle

%%%%%%%%%%%%%%%%%%%%%%%%%%%%%%%%%%%%%%%%%%%%%%%%%%%%%%%
%%%%%%%%%%%%%%%%%%%%%%%%%%%%%%%%%%%%%%%%%%%%%%%%%%%%%%%%%
\begin{abstract}
In this paper, we investigate a class of non-monotone reaction-diffusion equations with distributed delay and a homogenous boundary Neumann condition, which have a positive steady state. The main concern is the global attractivity of the unique positive steady state. To achieve this, we use an argument of a sub and super-solution combined with fluctuation method.  We also give a condition for which the exponential stability of the positive steady state is reached. As an example, we apply our results  to diffusive Nicholson blowflies and diffusive Mackey-Glass equation with distributed delay. We point out that we obtain some new results on exponential stability of the positive steady state for these cited models.
\end{abstract}

\section{Introduction}\label{sec:1}

\noindent In this paper, we  study the following  initial boundary  value problem
\begin{equation}\label{A}
\left \{
\begin{array}{lll}
 u_t(x,t)-\Delta u(x,t)=-f(u(x,t))+\int_0^{\tau}h(a)g(u(x,t-a))da,\;\;\ t>0, x\in \Omega \\
\dfrac{\partial u}{\partial n}(x,t)=0, \;\  x\in \partial \Omega, t>0,\\
u(x,t)=\phi(x,t),\;\  (x,t)\in \Omega\times[-\tau,0],
\end{array}
\right.
\end{equation}
where $\Omega$ is a bounded domain in $\mathbb{R}^n$ with smooth boundary $\partial \Omega$ and $\dfrac{\partial u}{\partial n}$ denotes the derivative along the outward normal direction on the boundary of $\Omega.$

 Recently, an increasing attention has been paid to local and non-local delay reaction-diffusion equations,  in  bounded and  unbounded domains, this class has a large field of applications, particulary in population dynamics,  see for instance \cite{Deng}, \cite{Thieme1}, \cite{Yi}, \cite{Yi1}, \cite{Zou}, \cite{Zou1}, \cite{Zou2} and references therein.

 \noindent Several previous results on the asymptotic behavior of solutions for this class of problem have been obtained by different  methods, (e.g.\cite{Xu}, \cite{Zou}, \cite{Zou1}, \cite{Xiao}). However, most of them suppose that, either $f$ is linear or $g$ is monotone.  In the context where $f$ is linear, Yi et al. \cite{Zou} established the relationship between the convergence of solutions of (\ref{A})  to the positive steady state and the convergence of the sequence defined by $x_{n+1}=g(x_{n}).$ The key condition for delay independent stability is that, the map $g$ does not have a true periodic-2 point (hence, not true periodic points of any period, by Sharkovski theorem).  For functional differential  equations this idea has been already used by many authors, see  \cite{Braveman}, \cite{Gyori}, \cite{Krisztin},  \cite{Lani}, \cite{Liz00}, \cite{Liz0},  \cite{Liz1}, \cite{Rost1}, \cite{Zou}  and the references therein; we particulary mention the leading works in \cite{Mallet-Nussbaum}, \cite{Mallet-Nussbaum1}, \cite{Mallet}.

  The main objective of this paper is to present an approach that unify the treatment of the global dynamics of solutions of (\ref{A}) for  non-monotone delayed term $g$ and linear or nonlinear function $f$.

\noindent To reach it, we will use an argument of sub and super-solutions to (\ref{A}). With the help of  comparison principles, which require some careful construction of deformations of the delayed term $g,$ we will prove the existence of an interval which attracts all solutions of (\ref{A}) and  where the delayed term is monotone, either nondecreasing or non-increasing.  From then, the global attractivity and exponential stability of the unique positive steady state are established by, principally, the fluctuation method.

\noindent Furthermore,  there is no result (at our knowledge) on exponential stability of the positive steady state of system (\ref{A}); we will give a sufficient condition for the exponential stability of the positive steady state.

Throughout this paper, we will make the following assumptions:

\noindent we suppose that the function $h$ is positive and    \[\int_0^{\tau}h(a)da=1.\]

\noindent \textbf{(T1)} $f$ and $g$ are Liptschitz continuous with $f(0)=g(0).$

\noindent \textbf{(T2)}  $g(s)> g(0)$ for all $s> 0$ and there exists a number  $B>0$ such that $\max\limits_{v\in[0,s]}g(v) <f(s)$ for all $s> B.$

\noindent We will also use the notation    \[K(s):=g(s)-f(s).\]

 Let $C=C(\bar{\Omega},\mathbb{R})$ and $X=C(\bar{\Omega}\times[-\tau,0],\mathbb{R})$ be equipped with the usual supremum norm $||.||.$ Also, let $C_{+}=C(\bar{\Omega},\mathbb{R}^{+})$ and $X_{+}=C(\bar{\Omega}\times[-\tau,0],\mathbb{R}^{+}).$ For any $\phi, \psi \in X,$ we write $\phi\geq \psi$ if $\phi-\psi \in X_+;$ $\phi>\psi$ if $\phi-\psi \in X_{+}\setminus\{0\};$ $\phi>>\psi$ if $\phi-\psi \in Int(X_+).$\\
We define the ordered intervals
\[[\phi,\psi]_{X}:=\{\xi \in X;  \phi\leq \xi\leq \psi\},\] and for any $\chi\in \mathbb{R},$ we write $\chi^{*}$ for the element of $X$ satisfying $\chi^{*}(x,\theta)=\chi$ for all $(x,\theta)\in \bar{\Omega}\times[-\tau,0].$ The segment $u_t\in X$ of a solution is defined by the relation $u_t(x,\theta)=u(x,t+\theta)$ for $x\in \bar{\Omega}$ and $\theta\in[-\tau,0].$  The family of maps \[U: \mathbb{R}^{+}\times X_{+}\rightarrow X_{+},\] such that \[(t,\phi)\rightarrow u_t(\phi)\] defines a continuous semiflow on $X_{+},$ \cite{Wu}. The map $U(t,.)$ is defined from $X_{+}$ to $X_{+}$ which is the semiflow $U_t,$ denoted by  \[U_t(\phi)=U(t,\phi).\]

The set of equilibria of the semiflow generated by (\ref{A}) is given by \[E=\{\chi^{*}\in X_{+}; \chi \in \mathbb{R}\; \mbox{and} \; g(\chi)=f(\chi)\}.\]

Our work is organized as follows: in the next section, we establish  existence, uniqueness and some estimates of the positive solution with the help of sub and super-solution. We also  prove that the unique positive steady state is globally attractive when the delayed  term is monotone non-decreasing.
Section 3 is devoted to investigating the non-monotone case; we will principally show, the existence of closed attractive intervals for solutions of (\ref{A}). In Section 4 we  will present some theorems related to global attractivity and exponential stability of the positive steady state.  Finally some examples are given to illustrate our theorems.

\section{Preliminaries}

Let $T (t)$ ($t\geq0$) be the strongly continuous semigroup of bounded linear operators on $C$
generated by the Laplace operator $\Delta$ under the homogenous Neumann  conditions. It is well known that
$T (t)$ ($t \geq 0$) is an analytic, compact and strongly positive semigroup on $C$. Define $F: X\rightarrow C$ by

\begin{equation}\label{F}
F(\phi)(x)=-f(\phi(x,0))+\int_0^{\tau}h(a)g(\phi(x,-a))da, \;\  \mbox{for all} \;\  x\in \bar{\Omega}.
\end{equation}

 We consider the following integral equation with the given initial data

\begin{equation}\label{AA}
\left \{
\begin{array}{lll}
u(t)=T(t)\phi(.,0)+\int_0^{t}T(t-s)F(u_s)ds,\quad t\geq 0,\\
u_0=\phi \in X.
\end{array}
\right.
\end{equation}

For each $\phi \in X,$ $u(.,t)$  with values in $C$ on its maximum interval $[0,\sigma_{\phi}),$ is called a mild solution of (\ref{A}), see for instance \cite{Fitzgibbon}, \cite{Martin0}, \cite{Martin1}, \cite{Wu}, and it is called classical if it is $C^2$ in $x$ and $C^{1}$ in $t.$

We now introduce the notion of a sub and super-solution of problem (\ref{A}).  Let us consider the following problems
\begin{equation}\label{A1+}
\left \{
\begin{array}{lll}
 \bar{u}_t(x,t)-\Delta\bar{u}(x,t)\geq  -f(\bar{u}(x,t))+\int_0^{\tau}h(a)g^{+}(\bar{u}(x,t-a))da,\;\;\ \;\ (x,t)\in \Omega\times(0,T], \\
\dfrac{\partial \bar{u}}{\partial n}(x,t)= 0, \;\  x\in \partial \Omega, \;\ t>0,\\
\bar{u}(x,t) =\phi^{+}(x,t),\;\  (x,t)\in \bar{\Omega}\times[-\tau,0],
\end{array}
\right.
\end{equation}
and
\begin{equation}\label{A1-}
\left \{
\begin{array}{lll}
 \underline{u}_t(x,t)-\Delta\underline{u}(x,t)\leq -f(\underline{u}(x,t))+\int_0^{\tau}h(a)g^{-}(\underline{u}(x,t-a))da,\;\;\ \;\ (x,t)\in \Omega\times(0,T], \\
\dfrac{\partial \underline{u}}{\partial n}(x,t)= 0, \;\  x\in \partial \Omega, \;\ t>0,\\
\underline{u}(x,t)=\phi^{-}(x,t),\;\  (x,t)\in \bar{\Omega}\times[-\tau,0],
\end{array}
\right.
\end{equation}
with $\phi^{-}$ and $\phi^{+}$ are nonnegative continuous function in $\bar{\Omega}\times[-\tau,0].$

\begin{Definition}
Let $g^{+},$ $g^{-}$ be Lipschitz continuous, monotone nondecreasing functions and assume that $g^{-}\leq g\leq g^{+}$ in $[a,b],$  $a, b \in \mathbb{R}^{+}.$  Suppose also that $a\leq \phi^{-}\leq \phi \leq \phi^{+}\leq b$ in $\bar{\Omega}\times[-\tau,0].$  The smooth bounded functions $\bar{u},$ $\underline{u}$  defined by (\ref{A1+}) and (\ref{A1-})  are called  super-solution and sub-solution of (\ref{A}) respectively.   The pair $(\bar{u}$, $\underline{u})$ is said to be ordered if $\underline{u}\leq \bar{u}$ in $\bar{\Omega}\times [-\tau,T].$
\end{Definition}

Throughout this paper, denote by $\bar{u}(g^{+},\phi^{+}):=\bar{u}(x,t;g^{+},\phi^{+})$ \big(respectively $\underline{u}(g^{-},\phi^{-}):=\underline{u}(x,t;g^{-},\phi^{-})$\big) a super-solution \big(respectively sub-solution) corresponding to delay-function $g^{+}$ \big(respectively $g^{-}$\big) and initial function $\phi^{+}$ in (\ref{A1+})\big(respectively $\phi^{-}$ in (\ref{A1-})\big) of problem (\ref{A}).

The following proposition is a particular case of Proposition 8.3.4 and Corollary 8.3.5 in \cite{Wu}.
\begin{Proposition}\label{pr0}
Let $\underline{u},$ $\bar{u}$ be ordered sub and super-solution of (\ref{A}), then (\ref{AA}) has a unique mild solution $u$ on $[0,\sigma_{\phi}),$ and this solution satisfies
\begin{equation*}
\underline{u}(g^{-},\phi^{-})\leq  u(x,t) \leq \bar{u}(g^{+},\phi^{+}), \;  \mbox{for} \;\ (x,t)\in \bar{\Omega}\times[0,\sigma),
\end{equation*}
with $\sigma$ is the largest time for which all these functions are defined.
\end{Proposition}

\begin{Lemma}\label{lem00}
If $\phi \in X_{+},$ the problem (\ref{AA}) admits a unique solution $u$. In addition we have the following results:
\begin{itemize}
\item{(i)} $(u)_t \in X_{+}$ for all $t\in[0,\sigma_{\phi});$

\item{(ii)} $\sigma_{\phi}=+\infty;$

\item{(iii)} $u(x,t)$ is a classical solution of (\ref{A}) for $x\in \bar{\Omega}$ and $t>\tau.$
\end{itemize}
\end{Lemma}
\begin{proof}
Let  $L\geq B,$ then for any $\phi \in[0,L]_{X},$  and $\underline{u}=0,$  the function $\underline{u}(g(0),\phi)$   is a sub-solution of (\ref{A}). Similarly, let ${g}^{+}(s):=\max\limits_{v\in[0,s]}g(v)$, for $\bar{u}=L,$ using \textbf{(T2)}, the function $\bar{u}(g^{+},\phi)$  is a super-solution of (\ref{A}). By Proposition \ref{pr0}, the problem (\ref{AA}) admits a unique solution  $0\leq u(x,t)\leq L$ for all $(x,t)\in \bar{\Omega}\times [0,\sigma_{\phi}).$ Since $L$ is arbitrarily large, it then follows that $\sigma_{\phi}=\infty$ for all $\phi \in X_{+}.$
This complete the proof of statements $(i)$ and $(ii)$.
The statement $(iii)$  follows from $(ii)$ and Theorem 2.2.6 in \cite{Wu}.
\end{proof}
We further have the following results.
\begin{Lemma}\label{lem01}
If $\phi \in X_{+}\setminus \{0\},$ then we have the following results:
\item{(i)} $(u)_t\in Int(X_{+})$ for all $t> 2\tau,$
\item{(ii)} The semiflow $U_t$ admits a compact attractor.
\end{Lemma}
\begin{proof}
 To prove the statement $(i)$ we first claim that $u_{\tau}\in X_{+}\setminus \{0\}.$ Otherwise, the Lemma \ref{lem00} implies that $u_{\tau}=0$ and thus $u(x,t)=0$ for all $(x,t)\in \bar{\Omega}\times [0,\tau].$ From (\ref{AA}) we have $\int_0^{t}T(t-s)F(u_s)ds=0$ for all $t\in [0,\tau].$ So, since $T(.)$ is a strongly positive semigroup we conclude that $F(u_s)=0$ for all $s\in[0,\tau].$ Moreover, since $u(x,t)=0$ for all $(x,t)\in \bar{\Omega}\times [0,\tau],$ and (\ref{F}), it follows
\begin{equation*}
\begin{array}{lll}
F(u_s)=-f(0)+\int_0^{\tau}h(a)g(u(.,s-a))da=0,
\end{array}
\end{equation*}
in view of \textbf{(T1)}, we get
\begin{equation*}
\int_0^{\tau}h(a)\bigg(g(u(.,s-a))-g(0)\bigg)da=0, \;\ \mbox{for all} \;\ s\in[0,\tau],
\end{equation*}
this implies that $g(\phi(.,\sigma))=g(0),$ and hence $\phi=0,$ which is absurd. Similarly we may show that $u_{2\tau}\in X_{+}\setminus \{0\}.$ Thus there exists $(x^{*},t^{*})\in \bar{\Omega}\times [\tau, 2\tau]$ such that $u(x^{*},t^{*})>0.$ It follows from (\ref{A}) and \textbf{(T1)}, \textbf{(T2)} that

\begin{equation*}
\begin{array}{lll}
 u_t(x,t)-\Delta u(x,t)&\geq& -f(u(x,t))+g(0),\\
 &=& f(0)-f(u(x,t)),\\
 &\geq & -L u(x,t) \quad \mbox{in}\;\ \Omega \times (t^{*},\infty),
 \end{array}
\end{equation*}
with $L$ is a Lipschitz constant associated to $f$. In addition
\begin{equation*}
\begin{array}{lll}
\dfrac{\partial u}{\partial n}=0 \quad \mbox{on}\;\ \partial \Omega\times(t^{*},\infty) \;\ \mbox{and}\\
u(x,t^{*})\geq 0, \;\ \mbox{ for all} \;\ x\in \Omega.
\end{array}
\end{equation*}
Let us introduce the following problem
\begin{equation*}
\left\{
\begin{array}{lll}
v_t(x,t)-\Delta v(x,t)=-Lv(x,t) \quad \mbox{in}\;\ (x,t)\in \Omega\times(t^{*},\infty),\\
\dfrac{\partial v}{\partial n}=0 \quad \mbox{on}\;\ \partial \Omega\times(t^{*},\infty),\\
v(x,t^{*})=u(x,t^{*}) \quad \mbox{for}\;\ x\in \bar{\Omega}.
\end{array}
\right.
\end{equation*}
By applying Theorem 7.3.4 in \cite{Smith}, we obtain that $u(x,t)\geq v(x,t)$ for all $(x,t)\in \bar{\Omega}\times(t^{*},\infty).$ On the other hand, by $v(x^{*},t^{*})>0$ and Theorem 7.4.1 in \cite{Smith}, we have $v(x,t)>0$ for all $(x,t)\in\bar{\Omega}\times(t^{*},\infty).$ So, $u(x,t)>0$ for all $(x,t)\in\bar{\Omega}\times(t^{*},\infty)$ and the statement (i) holds.
Concerning the statement (ii), let ${g}^{+}(s):=\max\limits_{v\in[0,s]}g(v),$ and $\bar{u}$ (its existence can be easily showed, see for instance \cite{Hale}) be a function verifies the following system

\begin{equation}\label{AA1}
\left \{
\begin{array}{lll}
 \bar{u}^{'}(t)=-f(\bar{u}(t))+\int_0^{\tau}{g}^{+}(\bar{u}(t-a))da,\;\;\ t>0, \\
\bar{u}(t)=\psi(t):=\max\limits_{x\in \bar{\Omega}}\phi(x,t),\;\  t\in [-\tau,0].
\end{array}
\right.
\end{equation}
  Observe that $\bar{u}(g^{+},\psi)$ is a super-solution of (\ref{A}). By Proposition \ref{pr0}, we have  $u(x,t)\leq \bar{u}(t)$ for all $(x,t)\in \bar{\Omega}\times[0,\infty).$  We claim that  $\limsup_{t\rightarrow \infty}\bar{u}(t)\leq B,$ for any $\phi \in X_{+}.$ Suppose, on the  contrary, that $\limsup\limits_{t\rightarrow \infty} \bar{u}(t):=l>B,$ then, in view of \textbf{(T2)} we have
\begin{equation}\label{V0}
-f(l)+{g}^{+}(l)<0.
\end{equation}
 On the other hand, from (\cite{Thieme0}, Proposition A.22) there exists $t_n\rightarrow \infty$ such that $\bar{u}(t_n)\rightarrow l$ and $\bar{u}'(t_n)\rightarrow 0$. Thus, substituting $\bar{u}(t_n)$ in (\ref{AA1}),
\begin{equation}\label{lim}
\bar{u}'(t_n)=-f(\bar{u}(t_n))+\int_0^{\tau}h(a){g}^{+}(\bar{u}(t_n-a))da,
\end{equation}
by the definition of $\limsup \bar{u}(t)$  and the monotonicity of the continuous function ${g}^{+}$ we have
\begin{equation}\label{lim1}
{g}^{+}(\lim\limits_{n\rightarrow \infty} \bar{u}(t_n-a)) \leq {g}^{+}(\lim\limits_{n\rightarrow \infty} \bar{u}(t_n)):={g}^{+}(l), \quad \forall a\in[0,\tau].
\end{equation}
Passing to the limit in (\ref{lim}) and combining with (\ref{lim1})   we get,
\begin{equation*}
0\leq -f(l)+{g}^{+}(l),
\end{equation*}
 and this provides a contradiction with (\ref{V0}). This implies that the semiflow $U_t: X_{+}\rightarrow X_{+}$ is point dissipative on $X_{+},$ see \cite{Smi-Thie}. Hence, applying  Theorem 2.2.6, \cite{Wu} and  Theorem 2.2.33, \cite{Smi-Thie}, we show that  $U_t,$  $t>\tau,$ admits a compact global attractor which
also attracts every bounded set in $X_{+}$. The statement $(ii)$ is reached. The proof of Lemma \ref{lem00} is completed.

\end{proof}

Although we will always suppose existence and uniqueness of the positive solution to the stationary problem of (\ref{A}), namely

\begin{equation}\label{stat}
\left \{
\begin{array}{lll}
 -\Delta U(x)=-f(U(x))+g(U(x)),\;\;\  x\in \Omega, \\
\dfrac{\partial U}{\partial n}(x)=0, \;\  x\in \partial \Omega, \\
\end{array}
\right.
\end{equation}
however, for the convenience of the reader, we give a lemma that ensures existence, uniqueness of positive solution to problem (\ref{stat}).
The proof of the following lemma is a special case of Theorem 2.3.4 in \cite{Pao}.
\begin{Lemma}
Let $\underline{u}(x),$ $\bar{u}(x)$ be ordered bounded positive sub and super-solution of (\ref{stat}) respectively and suppose that $\dfrac{g(u)-f(u)}{u}$ is a decreasing function for $u\in[\underline{u},\bar{u}],$  then the problem (\ref{stat}) admits a unique positive solution $u^{*}$ in $[\underline{u},\bar{u}].$
\end{Lemma}
\begin{remark}
Notice that a spatially inhomogeneous steady state solutions of reaction diffusion equations subject to Neumann conditions in a smooth domain are necessarily unstable see for instance \cite{Wu}.
\end{remark}

As a consequence of Theorem 9.3.3 in \cite{Wu}, we have the following asymptotic property of solution of (\ref{A}) in the case where $g$ is monotone nondecreasing.
\begin{Theorem}\label{th03}
Let $\underline{u}(x),$ $\bar{u}(x)$ (do not depend of time $t$) be bounded ordered positive sub and super-solution of (\ref{A}) respectively. Assume also that the problem (\ref{stat}) admits a unique positive solution $u^{*}$ in $[\underline{u},\bar{u}].$ If  $g$  is a nondecreasing function, then all solutions of (\ref{A}) converge to the positive steady state $u^{*}$.
\end{Theorem}

Suppose now that there exists $u^{*}$ such that
\begin{equation}\label{hypincr}
\left\{
 \begin{array}{lll}\vspace{0.3cm}
  g(u)> f(u)  \quad \mbox{for all } \; 0\leq u<u^{*}\\
 g(u)< f(u)  \quad \mbox{for all } \; u^{*}<u\leq B.
 \end{array}
 \right.
 \end{equation}

\begin{Theorem}\label{perincr}
 Assume that (\ref{hypincr}) holds and $g$ is a nondecreasing function. Suppose also that (\ref{stat}) admits a unique positive solution $u^{*}.$  Then the solution of problem (\ref{A})  is strongly persistent and converges to $u^{*},$  provided the corresponding initial function $\phi \in X_{+}\setminus\{0\}.$
\end{Theorem}
\begin{proof}
 First,  in view of (\ref{hypincr}) there exists $\varepsilon>0$ such that $g(\varepsilon)\geq f(\varepsilon).$ By Lemma \ref{lem01} (i) $u(x,t)>0$ for all $t> 2\tau$ and $x\in \bar{\Omega}.$ Now, suppose that $u(x,t)\geq \varepsilon$ for $(x,t)\in \bar{\Omega}\times[3\tau,4\tau],$ thus for $t>4\tau$ we can easily show that, for $\underline{u}(x,t)=\varepsilon$ if $(x,t)\in \bar{\Omega}\times (4\tau,\infty) $ and $\phi^{-}(x,t)=\varepsilon$ if $(x,t)\in \bar{\Omega}\times [3\tau,4\tau],$  the function $\underline{u}(g, \phi^{-})$ is a sub-solution of (\ref{A}) for $t\geq 3\tau.$  Finally, according to  Lemma \ref{lem01} (ii), we may choose $(\underline{u},\bar{u})=(\varepsilon, B).$ The proof is reached by  Theorem \ref{th03}.
\end{proof}

\section{The case where the delayed term $g$ is non-monotone }
\subsection{Persistence and estimates of solutions}

The aim of this section is to state some fundamental results, including the strong persistence and the closed attractive intervals for solutions of problem (\ref{A})
in the case where $g$ is non-monotone.

We make the following assumption, that will be used from now on.

\noindent   There exists a positive constant $u^{*}$ such that,

\begin{equation}\label{hyp0}
\left\{
\begin{array}{lll}\vspace{0.2cm}
\min\limits_{\sigma\in[s,u^{*}]}g(\sigma)>f(s),\quad \mbox{for} \quad 0<s<u^{*},\\
\max\limits_{\sigma\in [u^{*},s]}g(\sigma)<f(s),\quad  \mbox{for} \quad u^{*}<s\leq B,
\end{array}
\right.
\end{equation}

\begin{equation}\label{hyp000}
\begin{array}{lll}
  f'(0)\; \mbox{and} \; g'(0)\; \mbox{with} \;g'(0)>f'(0)>0\;
\mbox{and}\;
 f(s)>f(0),\quad \forall s>0,
 \end{array}
\end{equation}

Clearly,  $u^{*}$ is the unique positive value that satisfies $g(u^{*})=f(u^{*}).$\

With the aim to prove the strong persistence and to obtain the attractive intervals for solutions of  (\ref{A}), we need to construct nondecreasing functions having some  properties in order to  apply the results of the previous section. This is the goal of the  next lemmas. Their proofs are given in \cite{Touaoula}, for the reader convenience we provide them in details.

\begin{Lemma}\label{lem2}
Suppose that (\ref{hyp0}), (\ref{hyp000}) hold. Then there exists a positive constant $0<m<x^{*}$ satisfying
\begin{equation}\label{hyp3}
\left\{
\begin{array}{lll}\vspace{0.2cm}
g(s)> f(m)\quad \mbox{for} \quad m\leq s\leq B,\\ \vspace{0.2cm}
f(s)<f(m)\quad \mbox{for}\quad 0\leq s<m,\\
f(s)>f(m)\quad \mbox{for}\quad m<s\leq B,
\end{array}
\right.
\end{equation}
and $f,$ $g$ are strictly increasing over $[0,m].$
\end{Lemma}
\begin{proof}
 First, from (\ref{hyp000}) there exists $\sigma_0>0$ such that $f$ and $g$ are strictly increasing over $[0,\sigma_0].$ Let $\gamma>0$ be defined as
 \begin{equation}\label{m}
 \gamma:=\min\limits_{\sigma\in[\sigma_0,x^{*}]}f(\sigma)=f(\sigma_1).
 \end{equation}
  Since $f(\sigma_1)>f(0)$ then there exists $m_1\in (0,\sigma_0]$ such that $f(m_1)=f(\sigma_1).$ Note that for $\varepsilon>0$ so small, the positive constant $(m_1-\varepsilon)$ satisfies
\begin{equation*}
\left\{
\begin{array}{lll}\vspace{0.2cm}
f(s)< f(m_1-\varepsilon)\quad \mbox{for} \quad 0\leq s<m_1-\varepsilon,\\
f(s)>f(m_1-\varepsilon)\quad \mbox{for}\quad m_1-\varepsilon< s\leq B.
\end{array}
\right.
\end{equation*}
Indeed, since $f$ is strictly increasing over $[0,\sigma_0]$ and $m_1\leq \sigma_0$ it follows that $f(s)<f(m_1-\varepsilon)$  for all $s\in[0,m_1-\varepsilon)$ and  $f(s)>f(m_1-\varepsilon)$ for all $s\in(m_1-\varepsilon,\sigma_0].$ Further, for $\sigma_0\leq s\leq x^{*}$ and from (\ref{m}) we get

\begin{equation*}
f(s)\geq f(m_1)=f(\sigma_1),
\end{equation*}
thus,
\begin{equation*}
f(s)>f(m_1-\varepsilon).
\end{equation*}
Next, if $x^{*}<s\leq B$ so, in view of (\ref{hyp0})
\begin{eqnarray*}
f(s)&>&\max\limits_{\sigma\in[x^{*},s]}g(\sigma)\geq g(x^{*})=f(x^{*})>f(\sigma_1),
\end{eqnarray*}
from (\ref{m}), we have
\begin{eqnarray*}
f(s)>f(m_1)> f(m_1-\varepsilon).
\end{eqnarray*}
 We define  $\alpha:=\min\limits_{\sigma\in[m_1-\varepsilon,B]}g(\sigma).$ Hence, if $f(m_1-\varepsilon)<\alpha$ then  $m=(m_1-\varepsilon)$ satisfies (\ref{hyp3}), otherwise, there exists $m_2\in(0,m_1-\varepsilon]$ such that $f(m_2)=\alpha.$ Finally using the fact that $g$ is strictly increasing over $[m_2-\varepsilon,m_1-\varepsilon]$ and the first assertion of (\ref{hyp0}) we see that $m=(m_2-\varepsilon)$ satisfies (\ref{hyp3}). This completes the proof.
\end{proof}

Now  let us consider the following function,
  \begin{equation}\label{constr}
   g^{B}_m(s)=\left \{
\begin{array}{lll} \vspace{0.2cm}
g(s),& \quad \mbox{for}\;\ 0<s<\bar{m},\\
 f(m),& \quad  \mbox{for}\;\ \bar{m}<s\leq B,
\end{array}
\right.
\end{equation}
with $\bar{m}$ is the constant  satisfying $\bar{m}<m$ and $g(\bar{m})=f(m)$ and $m$ is defined in (\ref{hyp3}).\\
The following result is easily checked.
\begin{Lemma}\label{lem}
 Assume that (\ref{hyp0}), (\ref{hyp000}) hold. Then $g^{B}_m$ defined in (\ref{constr}) is a nondecreasing function over $(0,B)$ satisfying
 \begin{equation*}
   \left \{
\begin{array}{lll}\vspace{0.2cm}
g^{B}_m(s)\leq g(s), \quad 0\leq s\leq B,\\ \vspace{0.2cm}
 g^{B}_m(s)>f(s),\quad 0< s< m,\\
 g^{B}_m(s)<f(s),\quad m < s\leq B.
\end{array}
\right.
\end{equation*}
\end{Lemma}

Now we are in position to prove the strong persistence of solutions of (\ref{A}).
\begin{Lemma}
 Assume that (\ref{hyp0}), (\ref{hyp000})  hold.  Then the solution of problem (\ref{A}) is strongly persistent provided the corresponding initial data  $\phi \in X_{+}\setminus\{0\}.$
\end{Lemma}
\begin{proof}
 Set $g^{-}=g_{m}^{B}$ with $g^{B}_m$ is defined in  (\ref{constr}), then $\underline{u}(g^{-},\phi)$ is a sub-solution of (\ref{A}). In view of Proposition \ref{pr0}, we have $v(x,t)\leq u(x,t)$ for all $(x,t)\in \bar{\Omega}\times [0,\infty).$ Since $g^{B}_m$ is a  nondecreasing function, then the result is reached by applying Theorem \ref{perincr}.
\end{proof}

\noindent In the following,  we focus on functions $g$ having a maximum. More precisely, assume that the function $g$ satisfies :\\

There exists a positive constant $M$ such that,
\begin{equation}\label{hyp4}
\begin{array}{lll}
g(M)=\max\limits_{s\in \mathbb{R}^+}{g(s)}.
\end{array}
\end{equation}

We will investigate two cases, namely, $ u^{*}\leq M$ and $u^{*}>M.$

In order to state our next result we need the following lemma,
\begin{Lemma}\label{lem3}
Under the hypotheses (\ref{hyp0}), (\ref{hyp4}). Assume also that $u^{*}\in[0^{*},M^{*}]_{C}.$  Then the interval $[0^{*},M^{*}]_{X}$ attracts every solution $u$ of problem (\ref{A}), equivalently,  there exists $T>0$ such that
\begin{equation*}
0\leq u(x,t)\leq M, \;\ \forall  x\in \bar{\Omega}, \;\  t\geq T.
\end{equation*}
\end{Lemma}
\begin{proof}

Let $\bar{g}(s)=\max\limits_{\sigma\in[0,s]}g(\sigma),$ and we consider the following problem

 \begin{equation*}
 \left\{
\begin{array}{lll} \vspace{0.2cm}
v'(t)= -f(v(t))+\int_0^{\tau}h(a)\bar{g}\big(v(t-a)\big)da,\quad \mbox{for}\;\ t>0,\\
v(t)=\psi(t):=\max\limits_{x\in\bar{\Omega}}\phi(x,t) \quad \mbox{for}\;\ -\tau \leq t\leq 0.
\end{array}
\right.
\end{equation*}

Thus, $\bar{u}(g^{+},\psi)$ is a super-solution of (\ref{A}) with $\bar{u}=v$ and $g^{+}=\bar{g},$ Proposition \ref{pr0} leads to   $u(x,t)\leq v(t)$ for all $(x,t)\in \bar{\Omega}\times[0,\infty).$ Next we claim that there exists $T>\tau$ such that $v(t)\leq M$ for all $t\geq T.$ By contradiction, we suppose that there exists a positive constant $\bar{t}>T$ such that $v(\bar{t})=M$ and $v'(\bar{t})\geq 0,$ then, on one hand, and in view of (\ref{hyp0}), we have
\begin{equation}\label{M1}
g(M)\leq f(M).
\end{equation}
First suppose that $u^{*}<M$ so,
\begin{equation}\label{M11}
g(M)< f(M).
\end{equation}
  On the other hand,
   \begin{eqnarray*}
   0&\leq &-f(M)+\int_0^{\tau}h(a)\bar{g}(v(\bar{t}-a))da,
   \end{eqnarray*}
   consequently, we arrive at
  \begin{eqnarray*}
  f(M)\leq g(M),
   \end{eqnarray*}
   this is a contradiction with (\ref{M11}). Now if $u^{*}=M$ it follows that
   \begin{equation}\label{0}
   \max\limits_{v\in[0,s]}g(v)=g(u^{*}), \quad \forall s>u^{*}.
   \end{equation}
   Observe that, from the second assertion of (\ref{hyp0}) we get $g(u^{*})<f(s)$ for all $u^{*}<s\leq B,$ thus combining this with (\ref{0}) we conclude that
  \begin{equation}\label{M2}
   \max\limits_{v\in[0,s]}g(v)<f(s), \quad \forall s>u^{*}.
   \end{equation}
      Therefore according to Lemma \ref{lem01} (ii) (substituting the hypothesis in \textbf{(T2)} by (\ref{M2})) we show that  $\limsup\limits_{t\rightarrow \infty}v(t)\leq v^{*}.$  The result is reached similarly  as in the proof of Lemma \ref{lem01} (ii).
   \end{proof}

In the rest of this section we focus on $u^{*}>M.$ We impose some additional hypotheses on $f$ and $g$.

Assume that
\begin{equation}\label{hyp2}
\left\{
\begin{array}{lll}\vspace{0.3cm}
f(s)<f(M)&\quad \mbox{for}\quad 0\leq s<M,\\
f(s)>f(M)&\quad \mbox{for}\quad M<s\leq B.
\end{array}
\right.
\end{equation}

Now, to avoid any possibility of infinitely oscillation of $g$ around $f(M)$, we will assume that $g$ satisfies the next hypotheses: \

Setting
$$
\overline{\mathbb{D}}\equiv \{s\in [0, M]:\;g(s)=f(M)\},
$$

\begin{equation}\label{r}
\mbox{ there exists} \quad \bar{m}\in (0,M), \quad \mbox{such that} \quad  g(\bar{m})=f(M) \quad \mbox{and} \quad \bar{m}=\max \overline{\mathbb{D}}.
\end{equation}
Observe that $\overline{\mathbb{D}}\neq \emptyset$ since $g(0)=f(0)<f(M)<g(M).$

In the same way we define the set
$$
\underline{\mathbb{D}}\equiv \{s\in [M, B]:\;g(s)=f(M)\}.
$$
The rest of this subsection is devoted to estimating the solutions of (\ref{A}) into two different situations namely, either $\underline{\mathbb{D}}=\emptyset$ or  $\min \underline{\mathbb{D}}$ exists.  The following lemma  deals with the first case.

 \begin{Lemma}\label{lem1}
 Assume that $\underline{\mathbb{D}}=\emptyset$ and $u^{*}>M$. We also suppose that (\ref{hyp0}), (\ref{hyp000}),  (\ref{hyp2}), (\ref{r}) hold.
 Then the interval $[M^{*},B^{*}]_{X}$ attracts every solution $u$ of problem (\ref{A}).
 \end{Lemma}
 \begin{proof}
 It is clear that  $\underline{\mathbb{D}}=\emptyset$ implies,
  \begin{equation}\label{er}
 g(s)> f(M) \quad \mbox{for all}\quad s \in [M,B].
 \end{equation}
   In view of (\ref{r}) and the fact that $g(M)>f(M)$ observe that the value $\bar{m}$ defined in (\ref{r})  satisfies
   \begin{equation}\label{rr}
   \min\limits_{\sigma \in [\bar{m},M]}g(\sigma)=f(M),
   \end{equation}
otherwise $f(\bar{\sigma}):=\min\limits_{\sigma \in [\bar{m},M]}g(\sigma)<f(M)$ with $\bar{\sigma}\in(\bar{m},M).$ Further since $g(M)>f(M)$ then there exists $\bar{\bar{\sigma}} \in (\bar{\sigma},M)$ such that $g(\bar{\bar{\sigma}})=f(M),$ which contradicts (\ref{r}).

 Next, we introduce the following function,
 \begin{equation} \label{er1}
 g_M^{B}(s)=\left\{
 \begin{array}{lll}\vspace{0.3cm}
 \min\limits_{\sigma\in [s,M]}g(\sigma),& \quad \mbox{for} \quad 0<s<\bar{m},\\
 f(M),& \quad \mbox{for} \quad \bar{m}<s\leq B, \\
 \end{array}
 \right.
 \end{equation}
\vspace{0.2cm}\\

 we claim that the function $g^B_M$ is nondecreasing and satisfies,\vspace{0.2cm}

 \begin{equation}\label{gA}
 \left\{
 \begin{array}{lll}\vspace{0.2cm}
 g_M^{B}(s)\leq g(s),&\quad \mbox{for} &\quad 0\leq s \leq B,\\ \vspace{0.2cm}
  g_M^{B}(s)>f(s),&\quad \mbox{for} &\quad 0<s<M,\\
  g_M^{B}(s)<f(s),&\quad \mbox{for} &\quad M<s\leq B.
 \end{array}
 \right.
 \end{equation}
In fact, from (\ref{er}), (\ref{rr}) and (\ref{er1}) it is easily checked that $g_M^{B}$ is nondecreasing and $g_M^{B}(s)\leq g(s)$ for all $s\in[0,B],$
  further, we first take $0<s<\bar{m},$ then using the fact that $M<x^{*}$ we have
 \begin{eqnarray*}
 g_M^{B}(s)&=&\min\limits_{\sigma\in [s,M]}g(\sigma),\\
 &\geq & \min\limits_{\sigma\in [s,x^{*}]}g(\sigma),
 \end{eqnarray*}
 in view of (\ref{hyp0}) we get
 \begin{eqnarray*}
 g_M^{B}(s)>f(s).
 \end{eqnarray*}
 For $\bar{m}\leq s< M,$ the hypothesis (\ref{hyp2}) implies that,
 \begin{eqnarray*}
 g_M^{B}(s)=f(M)> f(s).
 \end{eqnarray*}
 Finally for $M<s\leq B$ using again (\ref{hyp2}) we obtain
 \begin{eqnarray*}
 g_M^{B}(s)=f(M)<f(s).
 \end{eqnarray*}

 Next, we consider the following problem,
  \begin{equation}\label{K2}
\left\{
\begin{array}{lll}\vspace{0.2cm}
v'(t)=-f(v(t))+\int_0^{\tau}h(a)g_M^{B}(v(t-a))da,&\quad t>0,\\
v(t)=\psi(t):=\min\limits_{x\in \bar{\Omega}}\phi(x,t),&\quad -\tau\leq t\leq 0.
\end{array}
\right.
\end{equation}

 So, $\underline{u}(g^{-},\psi)$ is a sub-solution of (\ref{A}) with $\underline{u}=v$ and $g^{-}=g_M^{B}.$  By Proposition \ref{pr0} we get $v(t)\leq u(x,t)\leq B$ for all $(x,t)\in \bar{\Omega}\times [0,\infty).$  Moreover, since  $g_M^{B}$ is a nondecreasing function and the problem (\ref{K2}) admits only $M$ as positive constant steady state,  then from Theorem \ref{th03}, $v(t)$ goes to $M$ as $t$ tends to infinity. As a conclusion,
$$\liminf\limits_{t\rightarrow \infty}u(x,t)\geq M \;\ \mbox{for all}\;\ x\in \bar{\Omega}.$$
In addition observe that  $[M^{*},B^{*}]_{X}$ is an invariant closed interval for the system (\ref{A}), that is, for $\phi \in [M^{*},B^{*}]_{X}$ we have $u\in [M^{*}, B^{*}]_{X}$ \big(this result is a direct consequence of Proposition \ref{pr0} and the fact that $\underline{u}(g^{-},M)$ with $\underline{u}=M$ and $g^{-}(.)=g_M^{B}(.)$ is a sub-solution of (\ref{A})\big). Finally,  the proof is completed by Lemma \ref{lem01} (ii).
\end{proof}

Now, we will give an interest to the case where $\min\underline{\mathbb{D}}$ exists.

\noindent   In this context, and throughout the rest of paper, we define the constant  $\mathbf{A}$ as

\begin{equation}\label{Ai}
    \mathbf{A}:=\min\big\{s\in(M,B],\; g(s)=f(M)\big\},
     \end{equation}
     and $M$ is defined in (\ref{hyp4}).  The number $\mathbf{A}$ plays a crucial role in estimating the solution of problem \ref{A}

     We suppose that
 \begin{equation}\label{Hypsupl}
\left\{
\begin{array}{lll}\vspace{0.3cm}
f(s)<f(\mathbf{A})&\quad \mbox{for}\quad u^{*}\leq s<\mathbf{A},\\
f(s)>f(\mathbf{A})&\quad \mbox{for}\quad \mathbf{A}<s\leq B.
\end{array}
\right.
 \end{equation}

     The following lemma gives the estimates of solutions of (\ref{A}).
 \begin{Lemma} \label{lm}
 Suppose that (\ref{hyp0}), (\ref{hyp000}), (\ref{hyp2}), (\ref{Hypsupl}) are fulfilled. Assume also that $u^{*}>M$ and
 \begin{equation}\label{cond}
 f(\mathbf{A})\geq  g(M).
 \end{equation}
 Then the interval $[M^{*},\mathbf{A}^{*}]_{X}$ attracts every solution $u$ of problem (\ref{A}).
 \end{Lemma}
 \begin{proof}
 First, we claim that $u^{*}<\mathbf{A}.$ Conversely,  suppose  $u^{*}\geq \mathbf{A},$ then if $u^{*}> \mathbf{A},$  due to the first assertion of (\ref{hyp0}) and $M<\mathbf{A}$, we obtain  $$ f(M):=g(\mathbf{A})\geq\min\limits_{\sigma \in [M,u^{*}]}g(\sigma)>f(M),$$ which is a contradiction. Further $u^{*}\neq \mathbf{A},$ if not, the  assertions in (\ref{hyp0}) give $$f(M)=g(\mathbf{A})=f(\mathbf{A})=f(u^{*}),$$
 however, $u^{*}>M$ and the claim is established by the second assertion of (\ref{hyp2}).

 On the other hand, from (\ref{Hypsupl}) and (\ref{cond}) we may  construct a  non-decreasing function $g^{+}$ over $[0,B]$ satisfying
 \begin{equation}\label{g+}
 \left\{
 \begin{array}{lll}\vspace{0.2cm}
 g^{+}(s)\geq g(s),&\quad \mbox{for} &\quad 0\leq s \leq B,\\ \vspace{0.2cm}
  g^{+}(s)>f(s),&\quad \mbox{for} &\quad 0<s<\mathbf{A},\\
  g^{+}(s)<f(s),&\quad \mbox{for} &\quad \mathbf{A}<s\leq B.
 \end{array}
 \right.
 \end{equation}

 Note that $\bar{u}(g^{+},\phi)$ with $\bar{u}=v$ is a super-solution of (\ref{A}) thus, it follows from  Proposition \ref{pr0}, that $u(x,t)\leq v(x,t)$ for all $(x,t)\in \bar{\Omega}\times [0,\infty).$ Therefore, Theorem \ref{th03} combined with (\ref{g+}) give
\begin{equation*}
\limsup\limits_{t\rightarrow \infty}u(x,t)\leq \mathbf{A}.
\end{equation*}
Further, as  $[0^{*},\mathbf{A}^{*}]_{X}$ is an invariant closed interval for (\ref{A}) \big($\bar{u}(g^{+},\mathbf{A})$ with $\bar{u}=\mathbf{A}$  is a super-solution of (\ref{A})\big); then  the result follows from Lemma \ref{lem01} (ii) . As a conclusion, there exists $T>0$ such that $0\leq u(x,t)\leq \mathbf{A}$ for all $(x,t)\in \bar{\Omega}\times (T,\infty).$
Next we turn to the lower bound of $u.$ From (\ref{hyp2}), (\ref{Ai})  the function $g_M^{\mathbf{A}}$ defined in (\ref{er1}) (by substituting $B$ by $\mathbf{A}$) is nondecreasing and satisfies (\ref{gA}). Now the function $v(g^{-},\phi)$ with $g^{-}=g_{\mathbf{A}}^{M}$ is a sub-solution of (\ref{A}), so $v(x,t)\leq u(x,t)$ for all $(x,t)\in \bar{\Omega}\times[0,\infty)$ and thus $\liminf\limits_{t\rightarrow \infty}u(x,t)\geq M$ for all $(x,t)\in \bar{\Omega}\times[0,\infty).$ Finally, notice that $[M^{*},\mathbf{A}^{*}]_{X}$ is an invariant closed interval for (\ref{A}), therefore Lemma \ref{lem01} (ii) completes the proof.

\end{proof}

\section{Global attractivity and exponential stability of the positive steady state}
\vspace{0.2cm}

In order to prove the global attractivity of the positive steady state,  from now on, we need the following hypotheses and definitions:

There exists a unique positive constant $M$ such that,
\begin{equation}\label{hyp44}
\begin{array}{lll} \vspace{0.2cm}
g(M)=\max\limits_{s\in \mathbb{R}^+}{g(s)},\mbox{and}\\
\mbox{the function}\; g\; \mbox{is nondecreasing  over}\; (0,M)\; \mbox{and nonincreasing over}\; (M,B).
\end{array}
\end{equation}

The following corollary deals with the case when $u^{*}\leq M$ and its proof is a direct consequence of Lemma \ref{lem3} and Theorem \ref{th03} .
\begin{Corollary}\label{cor1}
   Assume that (\ref{hyp0}), (\ref{hyp000}) hold. Then the positive steady state $u^{*}$ of problem (\ref{A}) is globally attractive provided that $u^{*}\leq M$ and $\phi\in X_{+}\setminus\{0\}.$
\end{Corollary}
Now we deal with situation where $u^{*}>M.$
According to Lemma \ref{lm}, and without a loss of generality we suppose that the initial function $\phi \in [M^{*},\mathbf{A}^{*}]_{X}.$

 Let $\hat{f}$ be the restriction of $f$ to  $[M,\mathbf{A}],$  and for $s\in [M,\mathbf{A}]$ define $G(s):={\hat{f}}^{-1}(g(s))$ if $\hat{f}$ is strictly increasing over $[M,\mathbf{A}].$  Notice that  for all $s\in[M,\mathbf{A}]$ the range of $g$ is contained in $[\hat{f}(M),\hat{f}(\mathbf{A})];$ in fact, for all $s\in[M,\mathbf{A}]$ and since $g$ is non-increasing over $[M,\mathbf{A}]$ then $g(\mathbf{A})\leq g(s)\leq g(M).$ In addition, from (\ref{Ai}) and (\ref{cond}) we have
\begin{equation*}
\hat{f}(M)\leq g(s)\leq \hat{f}(\mathbf{A}), \quad \mbox{for all} \;\; s\in [M,\mathbf{A}].
\end{equation*}
 As a conclusion the function $G$ is non-increasing and maps $[M,\mathbf{A}]$ to $[M,\mathbf{A}].$

 We now establish  the main theorem of this section related to the case  $u^{*}>M.$
\begin{Theorem}\label{th3}
Under the assumptions of Lemma \ref{lm},  the positive steady state of problem (\ref{A}) is globally attractive provided that  one of the following conditions holds:\
\begin{item}
 \item{\textbf{(H1)}} $\big(f(s)-f(0)\big)\big(g(s)-g(0)\big)$ is a nondecreasing function over $[M, \mathbf{A}]$.\\

 \item{\textbf{(H2)}} $f+g$ is a nondecreasing function over $[M, \mathbf{A}]$.\\

\item{\textbf{(H3)}} $\big(f(s)-f(0)\big)\big(g(s)-g(0)\big)+f+ g$ is a nondecreasing function over $[M, \mathbf{A}]$.\\

 \item{\textbf{(H4)}}  $\hat{f}$ is strictly increasing over $[M,\mathbf{A}]$ and $\dfrac{(GoG)(s)}{s}$ is a non-increasing function over $[M,x^*],$\\

\item{\textbf{(H5)}} $\hat{f}$ is  strictly increasing over $[M,\mathbf{A}]$ and $\dfrac{(GoG)(s)}{s}$ is a non-increasing function over $[x^*,\mathbf{A}],$
\end{item}
\end{Theorem}

\begin{proof}
\

\noindent \textbf{1/ Non-oscillatory case :} We first suppose that for all solutions $u$ of (\ref{A}) there exists $T>0$ such that $u(x,t)\leq u^{*}$ for all $x\in \bar{\Omega}$ and $t>T.$ Let $g^{-}$ be defined as

\begin{equation*}
g^{-}(s)=\left\{
\begin{array}{lll}\vspace{0.2cm}
\min\limits_{\sigma\in [s,u^{*}]} g(\sigma)&\quad \mbox{for} &\quad 0<s<u^{*},\\
g(u^{*}) &\quad \mbox{for} &\quad u^{*}\leq s\leq \mathbf{A},
\end{array}
\right.
\end{equation*}

 the function $\underline{u}(g^{-},\phi)$    is a sub-solution of (\ref{A}) ($g^{-}(s)\leq g(s)$ for all $s\leq u^{*}$); from Proposition\ref{pr0} $\underline{u}(x,t)\leq u(x,t)\leq u^{*}$ for all $(x,t)\in \bar{\Omega}\times[0,\infty).$ Further, in view of Theorem \ref{perincr}, $\underline{u}$ converges to $u^{*}$  as $t$ tends to infinity and $x\in \bar{\Omega}$. This means that each solution of (\ref{A}) converges to the positive steady state $u^{*}.$ Now if $u(x,t)\geq u^{*}$ for all $t\geq T,$ as above, let ${g}^{+}$ be defined as
\begin{equation}\label{p1}
{g}^{+}(s)=\left\{
\begin{array}{lll}\vspace{0.2cm}
\min\limits_{\sigma\in [s,u^{*}]} g(\sigma)&\quad \mbox{for} &\quad 0<s\leq u^{*},\\
\max\limits_{\sigma\in [u^{*},s]} g(\sigma)&\quad \mbox{for} &\quad u^{*}<s\leq \mathbf{A},
\end{array}
\right.
\end{equation}

then $\bar{u}(g^{+},\phi)$ is a super-solution of (\ref{A}) and so verify $u(x,t)\leq \bar{u}(x,t)$ for all $(x,t)\in \bar{\Omega}\times[0,\infty)$ and $\bar{u}(x,t)$ converges to $u^{*}$ as $t$ goes to infinity for all $x\in \bar{\Omega}$.

\noindent \textbf{2/ Oscillatory case} : We claim that this situation is not possible. Assume by contradiction that the solution $u$ oscillates infinitely around the positive steady state $u^*,$  we set $u^{+}(t)=\max\limits_{x\in\bar{\Omega}}u(x,t)$ and $u^{-}(t)=\min\limits_{x\in\bar{\Omega}}u(x,t).$ Denote $u^{\infty}=\lim\sup\limits_{t\rightarrow \infty}u^{+}(t)$ and  $u_{\infty}=\lim\inf\limits_{t\rightarrow \infty}u^{-}(t),$  then there exist sequences $\{t_n\}_n$, $\{s_n\}_n$ such that $(t_n,s_n)\rightarrow (+\infty,+\infty)$ as $n\rightarrow \infty$ and
\begin{equation*}
u^{+}(t_n)\rightarrow u^{\infty}, \;\ \frac{du^{+}}{dt}(t_n)=0, \;\ \mbox{as}\; n\rightarrow \infty,
\end{equation*}
and
\begin{equation*}
u^{-}(s_n)\rightarrow u_{\infty}, \;\ \frac{du^{-}}{dt}(s_n)= 0, \;\ \mbox{as}\; n\rightarrow \infty.
\end{equation*}

Let $x_n$ and $y_n$ be such that $u^{+}(t_n)=u(x_n,t_n),$ and $u^{-}(s_n)=u(y_n,s_n).$ Then
\begin{equation*}
\dfrac{\partial u}{\partial t}(x_n,t_n)=0, \;\ \dfrac{\partial u}{\partial t}(y_n,s_n)=0,
\end{equation*}
and
\begin{equation*}
\Delta u(x_n,t_n)\leq 0, \;\ \Delta u(y_n,s_n)\geq 0,
\end{equation*}
hence we get from (\ref{A}),
\begin{equation*}
\dfrac{\partial u}{\partial t}(x_n,t_n)-\Delta u(x_n,t_n)=-f(u(x_n,t_n))+\int_0^{\tau}h(a)g(u(x_n,t_n-a))da.
\end{equation*}

Passing to the limits and using the fact that $g$ is decreasing over $(M,A)$ we obtain
\begin{equation}\label{fluc}
f(u^{\infty})\leq g(u_{\infty}),
\end{equation}
and
\begin{equation}\label{fluc1}
f(u_{\infty})\geq g(u^{\infty}),
\end{equation}

from $ g(0)=f(0)$ we have,
\begin{equation}\label{fluc02}
f(u^{\infty})-f(0) \leq  g(u_{\infty})-g(0),
\end{equation}
 \begin{equation}\label{fluc03}
f(u_{\infty})-f(0) \geq g(u^{\infty})-g(0).
\end{equation}
Multiplying the expression (\ref{fluc02}) by $g(u^{\infty})-g(0)> 0$ and combining with (\ref{fluc03}) we obtain
 \begin{equation*}
 \big(f(u^{\infty})-f(0)\big)\big(g(u^{\infty})-g(0)\big) \leq \big(f(u_{\infty})-f(0)\big)\big(g(u_{\infty})-g(0)\big),
 \end{equation*}
 this fact together with the hypothesis $(H1)$ give $u^{\infty}\leq u_{\infty},$  so we reach a contradiction.
   Arguing as before we may conclude the results for $(H2)$ and $(H3)$.
Now suppose that $(H4)$ holds, in view of (\ref{fluc}), (\ref{fluc1}) and the monotonicity of $\hat{f}$ we arrive at,
 \begin{equation}\label{fluc2}
 u^{\infty} \leq G(u_{\infty}),
\end{equation}
and
 \begin{equation}\label{fluc3}
 u_{\infty} \geq  G(u^{\infty}).
\end{equation}
 Assume that $u_{\infty}<u^{*}\leq u^{\infty},$
then, applying the function $G,$ the inequalities (\ref{fluc2})-(\ref{fluc3})  become
\begin{equation*}
 u_{\infty} \geq G(u^{\infty}) \geq  (GoG)(u_{\infty}),
\end{equation*}
this gives,
\begin{equation}\label{r2}
\dfrac{(GoG)(u_{\infty})}{u_{\infty}}\leq 1=\dfrac{(GoG)(u^{*})}{u^{*}},
\end{equation}
due to $(H4)$ it ensures that $u^{*}\leq u_{\infty},$ which is a contradiction. Moreover, if $u_{\infty}\leq u^{*}<u^{\infty}$  then  from (\ref{r2}) we get
\begin{equation*}
u_{\infty}=u^{*}.
\end{equation*}
In addition, according to (\ref{fluc}) we have
\begin{equation*}
f(u^{\infty})\leq g(u_{\infty})=g(u^{*})=f(u^{*})<f(u^{\infty}),
\end{equation*}
 the contradiction is also reached. Using the same arguments as in $(H4)$ we establish the result for $(H5)$. The Theorem is proved.
\end{proof}

 Now let us investigate the global exponential stability of  the positive steady state. For this, we first establish the following lemma.
\begin{Lemma}\label{lem4}
Suppose that there exist two positive constants $\alpha,$ $\beta$ and a positive function $w$ such that
\begin{equation}\label{5}
w_t(x,t)-\Delta w(x,t)\leq -\alpha w(x,t)+\beta\int_0^{\tau}h(a)w(x,t-a)da.
\end{equation}
If $\alpha>\beta,$ then there exist two positive constants $C$ and $\gamma$ with $ \gamma \in (0,\alpha-\beta)$ such that
\begin{equation*}
w(x,t)\leq Ce^{-\gamma t},\;\ \forall (x,t)\in \bar{\Omega}\times [-\tau,\infty).
\end{equation*}
\end{Lemma}
\begin{proof}
By comparison principle, the result will be reached  if there exists $\gamma \in (0,\alpha-\beta)$ such that $z(t)=Ce^{-\gamma t}$ is solution of the following problem
\begin{equation}\label{z}
z'(t)= -\alpha z(t)+\beta\int_0^{\tau}h(a)z(t-a)da.
\end{equation}
Indeed by substituting the expression of $z$ in (\ref{z}) we get
\begin{equation*}
\beta\int_0^{\tau}h(a)e^{\gamma a}da+\gamma=\alpha,
\end{equation*}
thus for $F(\gamma):=\beta\int_0^{\tau}h(a)e^{\gamma a}da+\gamma$ we have
\begin{equation*}
F(\alpha-\beta)=\beta\bigg(\int_0^{\tau}h(a)e^{(\alpha-\beta)a}da-1\bigg)+\alpha>\alpha \;\ \mbox{and}\;\ F(0)=\beta<\alpha.
\end{equation*}
The Lemma is proved.
\end{proof}
Now we are in position to present our main result concerning the exponential stability of the steady state.
 \begin{Theorem}\label{th2}
 Under the conditions of  Theorem \ref{th3}. Suppose also that $f,$ $g$ are differential functions and $f$ is strictly increasing function on $[M,\mathbf{A}].$
 Then  the  positive steady state $u^{*}$ of (\ref{A}) is globally exponentially stable provided that
 \begin{equation}\label{hyp}
\underset{s\in [M,\mathbf{A}]}{\inf} f'(s)>  \underset{s\in [M,\mathbf{A}]}{\sup}|g'(s)|.
\end{equation}
 \end{Theorem}
 \begin{proof}
 First of all, in view of the proof of Theorem \ref{th3}, notice that the oscillatory case is not possible.  Thus, we set $v(x,t)=u(x,t)-u^{*},$ and we suppose that, there exists $T>0$ such that  $u(x,t)\geq 0$ for all $x\in \bar{\Omega}$ and $t\geq T$ (the proof will be the same if we assume that $y(t)\leq 0$ for all $t\geq T$).  Then $v$ satisfies the following problem
 \begin{equation*}
 v_t(x,t)-\Delta v(x,t)=-f'(\theta(t))v(x,t)+\int_0^{\tau}h(a)g'(\theta_1(t-a))v(x,t-a)da,
 \end{equation*}
 where $\theta(t)$ ($\theta_1(t-a)$ respectively) is a value between $u(x,t)$ and $u^{*}$ ($u(x,t-a)$ and $u^{*}$ respectively).
 Since $u(x,t)$ and $u^{*}$ belong to $[M, \mathbf{A}]$ we obtain,
 \begin{equation*}
 v_t(x,t)-\Delta v(x,t)\leq -\inf_{s\in[M,\mathbf{A}]} f'(s)v(x,t)+\sup_{s\in[M,\mathbf{A}]}|g'(s)|\int_0^{\tau}h(a)v(x,t-a)da.
 \end{equation*}
The result is established from (\ref{hyp}) and by applying  Lemma \ref{lem4} for $\alpha=\underset{s\in[M,\mathbf{A}]}{\inf} f'(s)$ and $\beta=\underset{s\in[M,\mathbf{A}]}{\sup}|g'(s)|$.
\end{proof}

\section*{Applications}
The goal of this section is to apply our results to two well-known models, namely Blowflies and  Mackey-Glass distributed delay equation. For a good survey in this direction see \cite{Berezansky} and references therein. For more details, concerning the results of stability of these both models, see \cite{Touaoula}.

\noindent First, observe that the condition  (\ref{hyp0}) is verified whenever the positive equilibrium exists. We set $\int_0^{\tau}h(a)da=1.$
\subsection*{Diffusive Nicholson's blowflies equation}
We consider The diffusive Nicholson's blowflies equation with distributed delay,
\begin{equation}\label{blow}
\begin{array}{lll}
u_t(x,t)-\Delta u(x,t)=-\delta u(x,t)+\int_0^{\tau}h(a)u(x,t-a)e^{-u(x,t-a)}da,\; x\in \Omega,\; t>0,
\end{array}
\end{equation}
together with homogenous Neumann condition.
For more details and results concerning the blowflies model, see \cite{Berezansky} and references therein.
As a direct application of our theorems, namely  Corollary \ref{cor1}, Theorem \ref{th3} (H4) and Theorem \ref{th2} we get the following results,
\begin{Theorem}
The unique positive steady state of (\ref{blow}) is globally attractive provided that
\begin{equation*}
1<\dfrac{1}{\delta}\leq e^{2}.
\end{equation*}
Moreover this positive steady state is globally exponentially stable if,
\begin{equation}\label{blow2}
  e<\dfrac{1}{\delta}< e^2.
\end{equation}
\end{Theorem}

\subsection*{Mackey-Glass model of hematopoiesis}
 The following diffusive Mackey-Glass equation with distributed delay is considered,
\begin{equation}\label{Mack}
u_t(x,t)-\Delta u(x,t) =-\delta u(x,t)+\int_0^{\tau}h(a)\dfrac{u(x,t-a)}{1+u^{n}(x,t-a)}da.
\end{equation}
with Neumann homogenous condition. For more results  of this type of equations, see \cite{Berezansky1} and the references therein.\

By the application of our  theorems stated in the previous sections, (Corollary \ref{cor1}, Theorem \ref{th3} (H4)-Theorem \ref{th2}) we obtain
\begin{Theorem}
Suppose that $\delta<1,$ then the positive steady state of (\ref{Mack}) is globally attractive if one of the following conditions is satisfied
\begin{equation}\label{h1}
0<n\leq 2,
\end{equation}

\begin{equation}\label{h22}
n>2\;\ \mbox{and}\;\ \dfrac{1}{\delta}<\dfrac{n}{n-2},
\end{equation}
Moreover the positive steady state of (\ref{Mack}) is globally exponentially stable if  the following condition is satisfied
\begin{equation*}
n>1\;\ \mbox{and}\;\ \dfrac{n}{n-1}<\dfrac{1}{\delta}<\dfrac{4n}{(n-1)^2}.
\end{equation*}
\end{Theorem}

\end{document}